\documentclass[12pt,a4paper,reqno]{amsart}

\usepackage{amssymb}

\usepackage{comment}

\allowdisplaybreaks

\newcommand{\C}{\mathbb{C}}

\newcommand{\PP}{\mathbb{P}}

\newcommand{\rto}{\dasharrow}

\newcommand{\I}{\mathcal{I}}

\newcommand{\cL}{\mathcal{L}}
\newcommand{\cO}{\mathcal{O}}

\newcommand{\cH}{\mathcal{H}}

\newcommand{\fS}{\mathfrak{S}}

\newcommand{\bin}{\binom}
\DeclareMathOperator{\rank}{rank}
\DeclareMathOperator{\Bs}{Bs}
\DeclareMathOperator{\Cr}{Cr}
\DeclareMathOperator{\Gr}{Gr}
\DeclareMathOperator{\Sym}{Sym}
\DeclareMathOperator{\Hilb}{Hilb}
\DeclareMathOperator{\im}{Im}
\DeclareMathOperator{\mult}{mult}
\DeclareMathOperator{\PGL}{PGL}

\DeclareMathOperator{\Bir}{Bir}
\newcommand\Biro{{\rm Bir}^{\circ}}
\newcommand\Bira[1]{{\rm Bir}^#1}

\newtheorem{theorem}{Theorem}
\newtheorem{lemma}[theorem]{Lemma}
\newtheorem{corollary}[theorem]{Corollary}

\theoremstyle{definition}
\newtheorem{definition}[theorem]{Definition}
\newtheorem{remark}[theorem]{Remark}
\newtheorem{example}[theorem]{Example}

\newtheorem{proposition}[theorem]{Proposition}

\title[On plane Cremona transformations of fixed degree]{On plane Cremona transformations of fixed degree}

\author{Cinzia Bisi, Alberto Calabri and Massimiliano Mella}
\address{Dipartimento di Matematica e Informatica\\
Universit\`a di Ferrara\\
\newline
\indent Via Machiavelli 35\\
 44121 Ferrara, Italia}
\email{cinzia.bisi@unife.it}
\email{alberto.calabri@unife.it}
\email{mll@unife.it}

\begin{document}

\maketitle

\section*{Introduction}

The birational geometry of algebraic varieties is governed by the
group of birational self-maps. It is in general very difficult to
determine this group for an arbitrary variety and as a matter of facts
only few examples are completely understood. The special case of the
projective plane attracted lots of attention since the ${\rm XIX^{\rm th}}$-century. The pioneering work of Cremona and then the classical geometers of the Italian and German school were able to give partial descriptions of it but it was only after Noether and Castelnuovo  that generators of the group were described. The Noether--Castelnuovo Theorem, see for instance \cite{A}, states that  the group of birational self-maps of $\PP^2$, usually called the \emph{plane Cremona group} and  denoted by $\Cr(2)$,  is generated by linear automorphisms of $\PP^2$ and a single birational non biregular map, the so-called \emph{elementary quadratic} transformation $\sigma\colon\PP^2\rto\PP^2$ defined by $\sigma([x:y:z])=[yz : xz : xy]$.

Even if the generators of $\Cr(2)$ have been known for a century now, many other properties of this group are still mysterious.
Only after decades, see \cite{Gi}, a complete set of relations has been described, and more recently the non simplicity of $\Cr(2)$ has been showed, \cite{CL}, and a good understanding of its finite subgroups has been achieved, see \cite{DI} and \cite{Bl}. This brief and fairly incomplete list is only meant to stress the difficulties and the large unknown parts in the study of  $\Cr(2)$, for a more complete picture the interested reader should refer to \cite{D} and \cite{desertibrasil}. Amid all its subgroups the one associated to  polynomial automorphisms of the plane, $Aut(\C^2)$, attracted even more attention than $\Cr(2)$ itself, \cite{CB}. The generators of $Aut(\C^2)$ are known since 1942, \cite{Ju},  and later on, \cite{Sh}, has been proved that $Aut(\C^2)$  is the amalgamated product of two of its subgroups, more precisely of the affine and elementary ones. Nevertheless this group  is not less mysterious and challenging than the entire Cremona group.
Jung's description yields a natural decomposition 
$$Aut(\C^2)=A\cup G[2] \cup G[3] \cup G[2,2] \cup G[4] \cup G[5] \cdots$$
into sets of  polynomial automorphisms of multidegree $(d_1, \cdots ,d_m)$ and  the affine subgroup $A$.
In \cite{FM}, Friedland and Milnor proved that $G[d_1, \cdots, d_m]$ is a smooth analytic manifold of dimension $(d_1 + d_2 + \cdots d_m +6).$
Later on, Furter, \cite{Fu1}, computed the number of irreducible components of polynomial automorphisms of $\mathbb{C}^2$ with fixed degree less or equal to $9$ and proved that the variety of polynomial automorphisms of the plane with degree bounded by the positive integer $n,$ is reducible when $n\geqslant 4.$
Later on Edo and Furter, \cite{Fu3}, studied some degenerations of the multidegrees: for example they were able to show that $G[3] \cap \overline{G[2,2]} \neq \emptyset$ using the lower semicontinuity of the length of a plane polynomial automorphism as a word of the amalgamated product, \cite{Fu2}.
Contrarily to what happens in the Cremona group $\Cr(2)$, see \cite{BF}, the group of polynomial automorphisms $G$ of the plane can be endowed with a structure of an infinite-dimensional algebraic group.
Denoted by $G_d$ the set of polynomial automorphisms of fixed degree $d,$
Furter proved in \cite{Fu4} that $G_d$ is a smooth and locally closed subset of $G.$ 

Inspired by these works  on $Aut(\C^2)$ and by \cite{CD} we  approach the study of $\Cr(2)$ forgetting its group
structure. Our aim is to study an explicit set
of functions generating a birational map. In this way the same birational map is
associated to different triples of homogeneous polynomials.

Let $f_1,f_2,f_3$ be homogeneous polynomials of degree $d$. Whenever $f_1,f_2,f_3$ are not all zero, let us denote by $[f_1:f_2:f_3]$ the equivalence class of $(f_1,f_2,f_3)$ with respect to the relation $(f_1,f_2,f_3)\sim(\lambda f_1,\lambda f_2,\lambda f_3)$, for $\lambda\in\C^*$.
Consider $[f_1:f_2:f_3]$ as an element of $\PP^{3N-1=3\binom{d+2}{2}-1=3d(d+3)/2+2}$, where the homogeneous coordinates in $\PP^{3N-1}$ are the coefficients of the $f_i$'s, up to multiplication by the same nonzero scalar. Setting
\begin{equation}\label{eq:gamma}
\gamma\colon\PP^2\rto\PP^2, \qquad \gamma([x:y:z])=[f_1(x,y,z):f_2(x,y,z):f_3(x,y,z)],
\end{equation}
let us define
\[
\Bir_d=\{ [ f_1 : f_2 : f_3 ] \mid \gamma \text{ is birational } \} \subset \PP^{3N-1}.
\] 
There is a natural subset to consider in $\Bir_d$, the one corresponding to triplets without common factors. The latter called  $\Biro_d$, that parametrizes birational maps of degree $d$, will be the main actor of this paper, see section \ref{S:defs} for all the relevant definitions.

Among other things, in \cite{CD} the authors describe $\Biro_d$ and $\Bir_d$ for $d\leqslant3$, see \S\ref{S:defs} for some of their results.
Their description is essentially based on a set-theoretic analysis of the plane curves contracted by a Cremona transformation of degree $\leqslant3$.
 
In this paper, we describe Cremona transformations $\PP^2\rto\PP^2$ in the usual classical way, i.e.\ we consider the so-called \emph{homaloidal net} of plane curves in the source plane corresponding to the net of lines in the target plane. Roughly speaking, this means to study the span $\langle f_1,f_2,f_3 \rangle\subset\C^{N=\binom{d+2}{2}}$ in the above setting.

By looking at base points of homaloidal nets, we prove the following, cf.\ \S\ref{main} later:

\begin{theorem}\label{thmBirod}
For each $d\geqslant2$, $\Biro_d$ is a quasi-projective variety of dimension $4d+6$ in $\PP^{3d(d+3)/2+2}$ and, if $d\geqslant4$, it is reducible.

Each irreducible component of $\Biro_d$ is rational.

Each element of the maximal dimension component of $\Biro_d$ is a \emph{de Jonqui\`eres transformation}, i.e.\ it is defined by a homaloidal net of plane curves of degree $d$ with a base point of multiplicity $d-1$ and $2d-2$ simple base points.

If $d\geqslant 4$, there are further irreducible components of $\Biro_d$ having dimension at most $2d+12$.
Indeed, each irreducible component of $\Biro_d$ is determined by a set of $(d-1)$-tuples of non-negative integers $(\nu_1,\nu_2,\ldots,\nu_{d-2},\nu_{d-1})$ such that $\sum_{i=1}^{d-1} i\nu_i=3d-3$ and $\sum_{i=1}^{d-1} i\nu_i^2=d^2-1$.
Each one of these components of $\Biro_d$ has dimension $8+2\sum_{i=1}^{d-1}\nu_i$ and its general element is a Cremona transformation given by a homaloidal net of plane curves of degree $d$ with $\nu_i$ base points of multiplicity $i$, in general position.
\end{theorem}

Quite surprisingly, in this set up, we may say that the {general Cremona transformation in $\Biro_d$ is \emph{de Jonqui\`eres}}.
For small values of the degree $d$ something more can be said. Our
approach yields an explicit description of $\Biro_d$ and allows us to
prove its connectedness, for $d\leqslant 6$.
\begin{theorem}
For each $d\leqslant6$, $\Biro_d$ is connected.
\end{theorem}

Even if, from our point of view, $\Biro_d$ is the right variety to
consider we complete the study extending the results to the whole $\Bir_d$.

\begin{theorem}\label{thmBird}
For each $d\geqslant2$, $\Bir_d$ is a connected quasi-projective variety of dimension
\[
\max \left\{ 4d+6,\ \frac{d(d+1)}{2}+7 \right\}
=\begin{cases}
4d+6   &\text{if $d\leqslant6$,} \\
\dfrac{(d+1)d}{2}+7  &\text{if $d\geqslant7$.}
\end{cases}
\]
in $\PP^{3d(d+3)/2+2}$.
If $d\geqslant3$, $\Bir_d$ is reducible.

Each irreducible component of $\Bir_d$ is rational.

If $d\geqslant7$, each element of the maximal dimension component of $\Bir_d$ is $[hl_1:hl_2:hl_3]$ where $\deg(h)=d-1$ and $\deg(l_i)=1$, $i=1,2,3$.

\end{theorem}

Also in this case our approach allows a more precise representation  of
$\Bir_d$ and an explicit  description of it for low $d$, see section \ref{main}.
This also corrects an imprecision in \cite{CD} about $\Bir_3$.
In principle one could describe the irreducible components of
$\Biro_d$ and of $\Bir_d$, for a fixed degree $d$. As a matter of
facts the combinatorial and computations become wild when the
degree increases. For this reason we think that it would be
interesting to find closed formulas for the numbers $N^\circ(d)$ and
$N(d)$ of irreducible components of $\Biro_d$ and $\Bir_d$ and to
study  the connectedness of $\Biro_d$ for each $d$. 

This study of $\Biro_d$, as pointed out by Cerveau and D\'eserti in \cite{CD},  also allows to build a
bridge between the classical algebraic geometry of $\Cr(2)$ and
foliations on $\mathbb{P}^k,$ $k\geqslant 2.$ The reconstruction of a foliation from its singular set,\cite{CO1},\cite{CO2}, the study of irreducible components of foliations on $\mathbb{P}^k,$ $k \geqslant 3,$ \cite{CN},\cite{CPV}, and the description of the orbits by the action of $\PGL(3,\mathbb{C})$ on the foliations of fixed degree 2 on $\mathbb{P}^2,$ \cite{CDGBM}, have been extensively studied and suggest the possibility to use our methods also in this context. We will not dwell on this here. However, we investigate the dynamical behaviour of plane Cremona transformations in a work in progress, see \cite{BCM}.

When this work was finished we were informed by J. Blanc of  D. Nguyen Dat's
 PhD thesis, \cite{Ng}. Some of the results, and the techniques, in this
paper overlap the content of his work. In particular he was able to compute the dimension of $\Biro_d$. Our approach allows to prove his Conjecture 3 about the irreducible components of $\Biro_d$ and corrects the wrong
statement on the non connectedness of $\mathcal H_d$ and $\Biro_d$, in \cite[Th\'eor\`eme 16, \S 6.2]{Ng}, for $d\geqslant 4$.

\bigskip
\noindent
ACKNOWLEDGEMENTS.
It is a pleasure to thank Ciro Ciliberto for the reference in Enriques--Chisini.
Furthermore, we thank the anonymous referee for a thoroughly reading, suggestions and remarks that improved the paper, in particular  for pointing out a mistake in Definition \ref{def:Hilb_nuI}.

\section{Notation, definitions and known results}\label{S:defs}

We work over the complex field.
Let $\C[x,y,z]_d$ be the set of homogeneous polynomials of degree $d$ in the variables $x,y,z$ with coefficients in $\C$, including the null polynomial.
In particular, $\C[x,y,z]_d\cong\C^{N}$ as $\C$-vector spaces, where
\[
N=\binom{d+2}{2}=\frac{d(d+3)}{2}+1
\]
is the number of coefficients of homogeneous polynomials of degree $d$ in 3 unknowns.

Whenever $f_1,f_2,f_3\in \C[x,y,z]_d$ are not all zero, let us denote by $[f_1 : f_2 : f_3]$ the equivalence class of the triplets $(f_1,f_2,f_3)$ with respect to the relation $(f_1,f_2,f_3)\sim(\lambda f_1,\lambda f_2, \lambda f_3)$, for $\lambda\in\C^*$.
Consider $[f_1:f_2:f_3]$ as an element of $\PP^{3N-1=3d(d+3)/2+2}$, where the homogeneous coordinates are all coefficients of the three polynomials $f_1,f_2,f_3$, up to multiplication by the same nonzero scalar for all of them.
Setting
\[
\gamma\colon\PP^2\rto\PP^2, \qquad \gamma([x:y:z]) = [f_1(x,y,z) : f_2(x,y,z) : f_3(x,y,z)],
\]
let us define, according to \cite{CD},
\[
\Bir_d=\{  [ f_1 : f_2 : f_3 ] \mid \gamma \text{ is birational }  \} \subset \PP^{3N-1}.
\]

In \cite{CD}, the integer $d$ is called the \emph{degree} of the map $\gamma$.
However, in algebraic geometry, if $f_1,f_2,f_3$ have a common factor, i.e.\ their greatest common divisor $\gcd(f_1,f_2,f_3)=h$ is a polynomial of positive degree $d'$, then the degree of the map $\gamma$ is usually considered as $d-d'$.
If $\gcd(f_1,f_2,f_3)=1$, the authors of \cite{CD} say that $\gamma$ has \emph{pure degree} $d$  and they define
\[
\Biro_d=\{  [ f_1 : f_2 : f_3 ] \in \Bir_d \mid \gamma \text{ has pure degree $d$ }  \} \subset \Bir_d\subset\PP^{3N-1}.
\]
If $[f_1:f_2:f_3]\in\Biro_d$, let us identify it with the birational map $\gamma$.

One may check that the subset $\{ [f_1:f_2:f_3] \mid \deg(\gcd(f_1,f_2,f_3)) > 0 \}$ of $\PP^{3N-1}$ is Zariski closed, e.g.\ by using resultants.

\begin{remark}\label{rem:Bir1}
When $d=1$, write the $3\times3$ nonzero matrix $(a_{ij})$, whose rows are the coefficients of $f_i$, $i=1,2,3$, thus $\Biro_1=\Bir_1=\PGL(3)$ is the Zariski open subset of $\PP^8$ where the determinant $\det(a_{ij})$ is nonzero and its Zariski closure is $\overline{\Bir_1}=\PP^8$.

Recall that $\PP^8\setminus\Bir_1=Z(\det(a_{ij}))$ is a cubic, irreducible hypersurface, which is singular along the locus where $\rank(a_{ij})=1$.
\end{remark}

In \cite{CD}, Cerveau and D\'eserti are interested  in $\Bir_d$ also for applications to the study of \emph{foliations}. Among other things, they prove the following:

\begin{theorem}[Cerveau-D\'eserti]
If $d=2$, the subset $\Biro_2$ $[\Bir_2$, resp.$]$ in $\PP^{17}$ is a smooth $[$singular, resp.$]$, unirational, irreducible, quasi-projective variety of dimension 14.
Their Zariski closures in $\PP^{17}$ coincide.

If $d=3$, the subset $\Biro_3\subset\PP^{29}$ is an irreducible, rationally connected, quasi-projective variety of dimension 18.
Furthermore, the subset $\Bir_3\subset\PP^{29}$ is singular and contains at least two irreducible components.
\qed
\end{theorem}

Actually, in \cite{CD}, the authors state that $\Bir_3$ has exactly two irreducible components.
We will see later, in Example \ref{Bir3}, that $\Bir_3$ has exactly three irreducible components.

Now we recall the classical notions which will be used later.

\begin{definition}\label{defgamma}
Let $\gamma=[f_1:f_2:f_3]\in\Biro_d$, $d\geqslant2$.
Setting $W$ the linear span $\langle f_1,f_2,f_3 \rangle$ of the polynomials $f_1,f_2,f_3$ in $\C^N$, the plane $\PP(W)\subset\PP^{N-1}$ is called the \emph{homaloidal net} associated to $\gamma$ and we denote it by $\cL_\gamma$.

The general element of $\cL_\gamma$ defines an irreducible rational
plane curve of degree $d$ passing through some fixed points $p_1,\ldots,p_r$ in
$\PP^2$, called \emph{set-theoretic base points} of $\cL_\gamma$, with certain multiplicities.
Let $\phi\colon S\to\PP^2$ be the blowing up at $p_1,\ldots,p_r$. The strict transform of the general element of $\cL_\gamma$ may have further base points on the exceptional curves in $S$, classically called \emph{infinitely near} base points, cf.\ \S2 in \cite{A}.
Equivalently, $\phi\circ\gamma$ may not be a morphism yet.
Eventually, there is a birational morphism $\phi'\colon S'\to\PP^2$ such that the strict transform of the net $\cL_\gamma$ is a complete 2-dimensional linear system on $S'$, or equivalently $\phi'\circ\gamma$ is a morphism.

Set $\nu_i$ the number of base
points of $\cL_\gamma$ with multiplicity
$i$, $i=1,\ldots,d-1$, including
infinitely near base points. The \emph{multi-index} of the homaloidal net
$\cL_\gamma$ is 
\begin{equation}\label{nuI}
\nu_I=(\nu_1,\nu_2,\ldots,\nu_{d-2},\nu_{d-1}).
\end{equation}
It is an easy consequence of the definition of homaloidal net
 (cf., e.g., \S2.5 in \cite{A}) that
\begin{align}\label{eq:Noether}
  &
\sum_{i=1}^{d-1} i^2\nu_i = d^2-1,
&&
\sum_{i=1}^{d-1} i\nu_i = 3(d-1),
\end{align}
and $\nu_{i}=0$, $i\geqslant d$.
Formulas \eqref{eq:Noether} are usually called Noether's Equations.

We say that the multi-index $\nu_I$
is the \emph{H-type} of the homaloidal net $\cL_\gamma$ and of the Cremona transformation $\gamma$.

One usually says that $\gamma$ is \emph{symmetric} if its H-type has a unique nonzero $\nu_i$.
By \eqref{eq:Noether}, $\gamma$ is symmetric if and only if either $d=2$ and $\nu_I=(3)$,
or $d=5$ and $\nu_I=(0,6,0,0)$, or $d=8$ and $\nu_I=(0,0,7,0,\ldots)$, or $d=17$ and the unique nonzero term in $\nu_I$ is $\nu_6=8$ (cf., e.g., Lemma 2.5.5 in \cite{A}).

Finally, let us define the \emph{index} $i(\gamma)$ of $\gamma$ as follows.
We set $i(\gamma)=0$ if and only if there is no infinitely near base point of $\cL_\gamma$, i.e.\ all base points of $\cL_\gamma$ belong to $\PP^2$.
If $\cL_\gamma$ has infinitely near base points, consider birational morphisms $\psi\colon S\to\PP^2$ such that the strict transform $\cL'$ of $\cL_\gamma$ in $S$ is a net with all base points belonging to $S$, i.e.\ $\cL'$ has no infinitely near base point.
Recalling that $\psi$ is the composition of finitely many blowing ups, each one at the maximal ideal of a single point, define $i(\gamma)$ as the minimum number of such blowing ups among all birational morphisms $\psi$ with the above property.
\end{definition}

\begin{remark}\label{r:nui1}
If $d\geqslant3$, the irreducibility of the general member of $\cL_\gamma$ implies that
$
\sum_{i>d/2} \nu_i \leqslant 1,
$
i.e.\ $\cL_\gamma$ has at most one base point of multiplicity $>d/2$.
\end{remark}

\begin{remark}\label{rem:m_i}
Usually, the H-type of a linear system $\cL$ is encoded by listing the multiplicities $m_i$ of the base points $p_i$ of $\cL$, which are commonly written in nonincreasing order, say $m_1\geqslant m_2\geqslant m_3\geqslant \cdots$, instead of a multi-index like $\nu_I$.
Given a multi-index $\nu_I=(\nu_1,\ldots,\nu_{d-1})$ of a homaloidal net $\cL_\gamma$, the multiplicities of the base points of $\cL_\gamma$, in nonincreasing order, can be computed as follows:
\begin{equation}\label{m_i}
m_i=\max\Biggl\{\, j : \sum_{k\geqslant j}\nu_k> i-1 \,\Biggr\},
\qquad
\text{for each $i\geqslant 1$}.
\end{equation}

Equivalently, set $m_1 =\max\{j: \nu_j\ne 0\}$ and then
\begin{align*}
m_1=m_2=\cdots=m_{R_1} &> m_{R_1+1}=\max\{j<m_1 : \nu_j\ne0\},
\quad R_1=\nu_{m_1},
\\
m_{R_1+1} = \cdots=m_{R_1+R_2} &> m_{R_1+R_2+1} = \max\{j<m_{R_1+1} : \nu_j\ne0\},
\quad R_2=\nu_{m_{R_1+1}},
\end{align*}
and so on, until $m_{R}=\min\{j:\nu_j\ne0\}$ with $R=\sum_{i=1}^{d-1} \nu_i=R_1+R_2+\cdots+R_s$, $s=|\{j:\nu_j\ne0\}|$.
\end{remark}

\begin{definition}\label{def:m_i}
Let $\nu_I=(\nu_1,\ldots,\nu_{d-1})$ be a multi-index with non negative entries.
We say that the positive numbers $m_1,\ldots,m_r$, with $r=\sum_{i=1}^{d-1}\nu_i$, computed in Remark \ref{rem:m_i} are the \emph{multiplicities associated to $\nu_I$}.
\end{definition}

\begin{definition}
One says that $\gamma\in\Biro_2$ is a \emph{quadratic} transformation. By \eqref{eq:Noether}, its H-type is $\nu_I=(3)$, namely $\gamma$ is defined by a homaloidal net of generically irreducible conics passing through three simple base points $p_1,p_2,p_3$.
We say that $\gamma$ is a quadratic transformation \emph{centered at $p_1,p_2,p_3$}.
\end{definition}

\begin{definition}\label{DeJon}
The map $\gamma\in\Biro_d$ is called a \emph{de Jonqui\`eres transformation} if there exists a base point of $\cL_\gamma$ with multiplicity $d-1$.
If $d\geqslant3$,  irreducibility and \eqref{eq:Noether} force $\nu_{d-1}=1$, $\nu_i=0$, $2 \leqslant i \leqslant d-2$, and $\nu_1=2d-2$, hence the H-type of a de Jonqui\`eres of degree $d\geqslant3$ is $(2d-2,0,\ldots,0,1)$.
\end{definition}

A little work on Equations \eqref{eq:Noether} proves the following (cf., e.g., Proposition 2.6.4 in \cite{A}):

\begin{theorem}[Noether's Inequality] \label{t:noetherineq}
Let $\gamma\in\Biro_d$, $d\geqslant2$, and let $m_1\geqslant m_2\geqslant m_3$ the maximal multiplicities of the base points of the homaloidal net $\cL_\gamma$.
Then
\begin{equation}\label{Noetherineq}
m_1+m_2+m_3\geqslant d+1.
\end{equation}
Furthermore, equality holds if and only if $\gamma$ is either symmetric or \emph{de Jonqui\`eres}.
\end{theorem}

The next, classical, lemma is very important for our purposes and we give here a proof only for the reader's convenience.

\begin{lemma}[{cf.\ \cite[p.\ 72]{Hu}}]\label{DeJonq}
If $d\geqslant3$, the maximal number of base points of a homaloidal net $\cL_\gamma$ is $\sum_{i=1}^{d-1} \nu_i=2d-1$ and, in that case, $\gamma$ is a \emph{de Jonqui\`eres} transformation.
\end{lemma}

\begin{proof}
In \eqref{eq:Noether}, multiply by $d$ the second equation and then subtract the first one, thus
\begin{equation}\label{2d-1}
(d-1)(2d-1)=\sum_{i=1}^{d-1} i(d-i)\nu_i,
\end{equation}
On the real interval $[1,d-1]$,
the function $g(x)=x(d-x)$ has a maximum at $x=d/2$ and a minimum at $x=1$ and at $x=d-1$, so $1(d-1)\leqslant g(x) \leqslant d^2/4$.
Hence Equation $\eqref{2d-1}$ implies $2d-1 \geqslant \sum_{i=1}^{d-1}\nu_i$.
Equality holds if and only if $\nu_{1}+\nu_{d-1}=2d-1$ and $\nu_i=0$, $i=2,\ldots,d-2$.
Remark \ref{r:nui1}, Formulas \eqref{eq:Noether} and the hypothesis $d\geqslant3$ force $\nu_{d-1}=1$.
\end{proof}

\begin{remark}
In algebraic geometry, when a family of objects $\{X_q\}_{q\in\Sigma}$, like varieties, maps, etc., is parametrized by the points of an \emph{irreducible} algebraic variety $\Sigma$, one usually says that the \emph{general} object $X$ has a certain property $P$ if the subset of points $q\in\Sigma$, such that $X_q$ has the property $P$, contains a Zariski open dense subset of $\Sigma$.
\end{remark}

Later we will use finite sets of points in $\PP^2$ and in its symmetric products.

\begin{definition}[{cf.\ \cite{M}}]
The \emph{$r$-th symmetric product} of $\PP^2$, denoted by $\Sym^r(\PP^2)$, is the quotient $(\PP^2)/\fS_r$ of the Cartesian product $(\PP^2)^r$ under the action of the symmetric group $\fS_r$ permuting the factors.
One has that $\Sym^1(\PP^2)=\PP^2$, if $r=1$, otherwise
$\Sym^r(\PP^2)$ is a rational, singular, irreducible, projective variety of dimension $2r$.
If $p_1,\ldots,p_r\in\PP^2$ are distinct points, then $\Sym^r(\PP^2)$ is smooth at the class of $(p_1,\ldots,p_r)$.
\end{definition}

A useful tool to take care of infinitely near points is the Hilbert scheme.

\begin{definition}[cf.\ \cite{F1, F2, I}]
The Hilbert scheme $\Hilb^r(\PP^2)$ parametrizes zero-di\-men\-sional subschemes of $\PP^2$ with \emph{length} $r$, i.e.\ whose Hilbert polynomial has degree 0 and it is equal to $r$.
It is a rational, smooth, irreducible, projective variety of dimension
$2r$ and it is a desingularization of $\Sym^r(\PP^2)$.
\end{definition}

Recall that $r$ points in
general position in $\PP^2$ determine an open dense subset of either $\Hilb^r(\PP^2)$ or  $\Sym^r(\PP^2)$.

 \begin{definition}\label{def:Hilb_nuI} Let $d$ be a
   positive integer. A multi-index
   $\nu_I:=(\nu_1,\nu_2,\ldots,\nu_{d-2},\nu_{d-1})$ with
   \emph{length} $\rho$ and \emph{reduced length} $r$ 
   is a $(d-1)$-tuple of non negative
   integers with $r=\sum_{i=1}^{d-1} \nu_i$ and $\rho=\sum_{i=1}^{d-1}
   {i(i+1)}\nu_i/2$.
Setting $I^*=\{i\mid \nu_i>0\}$,
let us define
\[
\Hilb^{\nu_I}(\PP^2)=
\prod_{i\in I^*} \Hilb^{\nu_i}(\PP^2),
\]
thus a point $Z\in\Hilb^{\nu_I}(\PP^2)$ is $Z=([Z_i])_{i\in I^*}$ with $[Z_i]\in\Hilb^{\nu_i}(\PP^2)$, for each $i\in I^*$.

Let us denote by $\Hilb^{\nu_I}_\bullet(\PP^2)$ the dense open subset of $\Hilb^{\nu_I}(\PP^2)$ whose elements are $Z=([Z_i])_{i\in I^*}$ such that the zero-dimensional scheme
$
Z^r=\cup_{i\in I^*} Z_i
$
is supported on $r$ distinct points in $\PP^2$, in particular $Z_i$ is a collection of points $p_{i,1},\ldots,p_{i,\nu_i}$ in $\PP^2$.
To a point $Z\in\Hilb^{\nu_I}_\bullet(\PP^2)$, we associate  the  zero scheme $[Z_{\nu_I}]$ in $\Hilb^\rho(\PP^2)$ with 
   $(Z_{\nu_I})_{\rm red}=Z^r$ and $Z_{\nu_I}$  given by the union of the points  $p_{i,j}$ with multiplicity $i$, for each $i\in I^*$ and $j=1,\ldots,\nu_i$.
 \end{definition}

\begin{remark} 
Note that $[Z_{\nu_I}]$  is not a general point of
  $\Hilb^\rho(\PP^2)$. This is the main motivation to introduce both length and reduced length. Equations
  (\ref{eq:Noether}) force  the length $\rho$ of $\nu_I$ to be uniquely determined by $d$, namely $\rho=(d+4)(d-1)/2$.
\end{remark}

\begin{definition}
We say that a multi-index 
$\nu_I=(\nu_1,\nu_2,\ldots,\nu_{d-2},\nu_{d-1})$
is \emph{admissible} if there is an element
$Z\in \Hilb^{\nu_I}_\bullet(\PP^2)$ such that the linear system
$\Lambda_Z:=|\I_{Z_{\nu_I}}(d)|$ is non empty, of
the expected dimension
$$\dim\Lambda_Z=\bin{d+2}2-\sum_{i=1}^{d-1} \frac{i(i+1)\nu_i}2$$
and the general element of $\Lambda_Z$
is an irreducible curve. In such a case
we say that $Z$ is an \emph{admissible cycle}
associated to $\nu_I$.
\end{definition}

The following theorem is classical, cf., e.g., \cite[V.II.20]{EC} and \cite[Theorem 5.1.1]{A}.

\begin{theorem}
Fix a positive
  integer $d$ and a multi-index
  $\nu_I$. Assume that $\nu_I$ satisfies
  Equations (\ref{eq:Noether}). Let
  $Z\in \Hilb^{\nu_I}_\bullet(\PP^2)$ be a point. Then $\nu_I$
  and $Z$ are admissible if and only
  if $\Lambda_Z$ is a homaloidal net.
\label{th:classical}
\qed
\end{theorem}

\begin{remark}\label{counterexample}
There are finitely many multi-indexes
$\nu_I$ satisfying Equations
\eqref{eq:Noether}, when $d$ is fixed,
but not all such multi-indexes are H-types,
i.e.\ do give rise  to homaloidal nets in $\Biro_d$.
The first, classical, example is $(6,0,2,0)$, $d=5$. In this case the
line through the two triple points is a fixed component of the linear
system of 5-ics with the assigned multiplicities. The movable part
of the net is given by elliptic quartic curves. In particular the
point representing it in $\PP^{62}$ is not in $\Bir_5$.
\end{remark}

To get rid of this behaviour we proceed as follows, cf.\ Hudson's test in \cite{Hu} and in \cite{A}.

\begin{definition}\label{def:nuirr}
Let $\nu_I=(\nu_1,\ldots,\nu_{d-1})$, $d\geqslant2$, be a multi-index satisfying Equations \eqref{eq:Noether}.
When $d=2$, we say that $\nu_I=(3)$ is \emph{1-irreducible}.
Suppose then $d\geqslant3$ and let
$
m_1\geqslant m_2\geqslant m_3
$
be the maximal multiplicities associated to $\nu_I$, cf.\ Definition \ref{def:m_i}.

We say that $\nu_I$ is \emph{1-irreducible} if $m_1+m_2\leqslant d$.
Let
$
d'=2d-m_1-m_2-m_3.
$
By \eqref{eq:Noether} and $d\geqslant 3$, the same proof of Noether's inequality \eqref{Noetherineq} in \cite{A} shows that $d> d' \geqslant 2$.

Now define a new multi-index $q(\nu_I)=(\nu'_1,\nu'_2,\ldots,\nu'_{d'-2},\nu'_{d'-1})$ by the following steps:
\begin{itemize}
\item for each $j=1,2,3$, decrease $\nu_{m_j}$ by 1\\
(if $m_3=m_2=m_1$, this means to decrease $\nu_{m_1}$ by 3);

\item set $\varepsilon=d-m_1-m_2-m_3$;

\item for each $j=1,2,3$, set $k=m_j+\varepsilon$ and, if $k>0$, increase $\nu_k$ by 1;

\item finally, for each $i=1,\ldots,d'-1$, set $\nu'_i=\nu_i$.
\end{itemize}
We say that $\nu_I$ is \emph{irreducible} if $\nu_I$ is 1-irreducible, $q(\nu_I)$ is 1-irreducible, $q(q(\nu_I))$ is 1-irreducible, and so on, for all new multi-indexes until one stops, when $d$  becomes $2$.
A script, which runs this irreducibility test, is listed in the appendix.
\end{definition}

\begin{remark}\label{q()}
If $\nu_I$ is the H-type of a Cremona transformation $\gamma\in\Biro_d$, and the maximal multiplicities $m_1\geqslant m_2\geqslant m_3$ of the base points of the homaloidal net $\cL_\gamma$ occur at three points $p_1, p_2, p_3$ such that a quadratic transformation $\omega$ centered at $p_1,p_2,p_3$ is well-defined, then $q(\nu_I)$ is just the H-type of $\omega\circ\gamma\in\Biro_{d'}$.

Setting $p_4,\ldots,p_r$ the other base points of $\cL_\gamma$, with respective multiplicities $m_4\geqslant \cdots \geqslant m_r$, one has (cf., e.g., Corollary 4.2.6 in \cite{A}) that the multiplicities of the homaloidal net of  $\omega\circ\gamma$ at the points corresponding via $\omega$ to $p_1,\ldots,p_r$  are respectively $m'_1, m'_2, \ldots, m'_r$  where
\begin{align*}
&
m'_i=m_i-\varepsilon,\ i=1,2,3,
&&
\varepsilon=m_1+m_2+m_3-d,
&&
m'_j=m_j,\ j\geqslant 4,
\end{align*}
Note that $d'=d-\varepsilon$,
$m'_1=d-m_2-m_3$, $m'_2=d-m_1-m_3$ and $m'_3=d-m_1-m_2$.

The same formulas hold even if $m_1, m_2, m_3$ are not the maximal multiplicities of the base points of $\cL_\gamma$. Note that $\varepsilon > 0$ if and only if $m_1+m_2+m_3>d$.
\end{remark}

The next theorem appears to be classical and it has been implicitly used by Hudson, but it had probably fallen into oblivion, cf.\ historical remark 5.3.6 in \cite{A} and the references therein.
For a modern proof, see Theorems 5.2.19 and 5.3.4 in \cite{A}.

\begin{theorem}
Fix an integer $d$ and a multi-index
$\nu_I$ which satisfies \eqref{eq:Noether} and is irreducible, according to Definition \ref{def:nuirr}.
Setting $r=\sum_{i=1}^{d-1}\nu_i$, let $m_1\geqslant m_2 \geqslant \cdots \geqslant m_r$ be the multiplicities associated to $\nu_I$, cf.\ Definition \ref{def:m_i}.

Then $\nu_I$ is admissible and there exists a non-empty Zariski open subset $U$ of $(\PP^2)^r$ (that is $\PP^2\times \cdots \times \PP^2$, $r$ times) such that for each $(p_1,\ldots,p_r)\in U$ there exists a Cremona transformation of degree $d$ which has $p_i$ as base point of multiplicity $m_i$, $i=1,\ldots,r$, and has no other base points.
\label{th:classical_inverse}
\qed
\end{theorem}

In our notation, this translates as follows.

\begin{corollary}\label{cor:classical_inverse}
In the hypothesis of the previous theorem,
there exists an open dense subset $U_{\nu_I}\subset \Hilb^{\nu_I}(\PP^2)$  such that for any point $Z\in U_{\nu_I}$, $Z$ is associated to $\nu_I$ and  the linear system $\Lambda_Z$ is a homaloidal net.
\qed
\end{corollary}

The algebraic structure of $\Bir_d$ and $\Biro_d$ is already known, cf.\ \cite[Lemma 2.4]{BF}.

\begin{lemma} \label{quasiproj}
The subsets $\Bir_d$ and $\Biro_d$  in $\PP^{3N-1}$ are quasi-projective varieties.
\qed
\end{lemma}

We will use also planes, i.e.\ linear subspaces of dimension 2, in $\PP^{N-1}$, or equivalently three-dimensional vector subspaces of $\C^N$.
A convenient setting is Grassmannians.

\begin{definition}
Denote by $\Gr(3,N)$ the Grassmannian variety parametrizing three-dimensional vector subspaces of $\C^N$, i.e.\ planes in $\PP^{N-1}$.
It is a smooth, irreducible, rational, projective variety of dimension $3(N-3)$.
\end{definition}

\section{Cremona transformations of fixed degree}\label{main}

In this section we use notation introduced in \S\ref{S:defs}. Fix $d$ a positive integer.

\begin{remark}\label{r:disjoint}
For each $d\geqslant2$,
there is a one-to-one map
\[
\Bir_d\setminus\Biro_d \to \coprod_{a=1}^{d-1} \left( \PP(\C[x,y,z]_a) \times \Biro_{d-a}\right),
\qquad [f_1:f_2:f_3] \mapsto \left( h , \left[ \frac{f_1}{h} : \frac{f_2}{h} : \frac{f_3}{h} \right] \right),
\]
where $h=\gcd(f_1,f_2,f_3)$.
The inverse is the collection of the maps
\begin{equation}\label{tau}
\tau_a\colon\PP(\C[x,y,z]_a) \times \Biro_{d-a} \to \Bir_{d},
\qquad
\tau_a(h,[f_1:f_2:f_3]) = [hf_1:hf_2:hf_3],
\end{equation}
which are injective, for each $a=1,\ldots,d-1$.
\end{remark}

Let us focus on $\gamma=[f_1:f_2:f_3]\in\Biro_d$, in particular $\gcd(f_1,f_2,f_3)=1$.
Let $\cL_\gamma=\PP(\langle f_1,f_2,f_3\rangle)$ be the homaloidal net associated to $\gamma$,
recalled in \S\ref{S:defs}.

\begin{definition}\label{Bs}
We denote by ``$\Bs$'' the map which sends a Cremona transformation $\gamma\in\Biro_d$ to its base locus:
\begin{equation}\label{eq:baselocus}
\Bs\colon\Biro_d \to \Hilb(\PP^2),
\qquad
\gamma=[f_1:f_2:f_3] \mapsto \Bs(\gamma)=Z(f_1,f_2,f_3),
\end{equation}
where $\Bs(\gamma)=Z(f_1,f_2,f_3)$ is a 0-dimensional subscheme of $\PP^2$.
\end{definition}

\begin{remark}
The map $\Bs$ is algebraic. To see this,
fix an irreducible component $A\subset \Biro_d$, take
\[
I_c=\{ ([f_1:f_2:f_3],[Z]) \mid Z\subset Z(f_1,f_2,f_3) \} \subset A\times \Hilb^c(\PP^2),
\]
and let $p_1,p_2$ be the projections on the two factors.
By Equations \eqref{eq:Noether}, each homaloidal net in $\Biro_d$ has the expected dimension.
Hence, for each $A$ there is a unique $c$ such that $p_1^{-1}([f_1:f_2:f_3])\cap I_c$ is a point.
In this setting, $\Bs=p_2\circ p_1^{-1}.$
\end{remark}

There is no inverse map to $\Bs$ in
\eqref{eq:baselocus}, because
the map $\Bs$ is not injective. The definition
of $\Biro_d$ is such that two different
bases of the same homaloidal net gives
different elements in  $\Biro_d$ whilst
have the same base locus.
This is quite awkward, at least from the
algebraic geometry point of view, but
can be, somehow, settled as follows.

\begin{lemma}\label{w}
Let $\gamma=[f_1:f_2:f_3]$ and
$\delta=[g_1:g_2:g_3]$ be in $\Biro_d$
be such that $\Bs(\gamma)=\Bs(\delta)$. Then
$
\langle f_1,f_2,f_3\rangle=\langle g_1,g_2,g_3\rangle=W\subset\C^{d(d+3)/2+1},
$
and there exists a unique change of basis matrix $\omega\in\PGL(3)$ such that the triplet $(g_1,g_2,g_3)=(f_1,f_2,f_3)\omega$ in $W\cong\C^3$.
\end{lemma}

\begin{proof}
Let  $Z_\gamma=(\Bs(\gamma))_{\rm red}=(\Bs(\delta))_{\rm red}$ be the reduced base locus of $\gamma$ and $\delta$.
Theorem \ref{th:classical} says that $\dim \Lambda_{Z_\gamma}=2$ and thus
$W=\langle f_1,f_2,f_3\rangle\ni g_i$, $i=1,2,3$.
Hence, by definition,
the $f_i$'s and the $g_i$'s are two
bases of the 3-dimensional vector space
$W\subset \C^{d(d+3)+1}$ and there is a
unique change of basis $\omega$ sending one to the
other. To conclude, observe that
$[f_1:f_2:f_3]$ and $[\lambda
f_1:\lambda f_2:\lambda f_3]$, $\lambda\in\C^*$, represent
the same element in $\Biro_d$.
\end{proof}

The previous lemma suggests how to change
the target space in \eqref{eq:baselocus}
in order to get a birational map.  
For this purpose, we construct suitable morphisms.

Fix a positive integer $d$ and an irreducible (according to Definition \ref{def:nuirr}) multi-index $\nu_I$, satisfying \eqref{eq:Noether}. By Corollary \ref{cor:classical_inverse} there is a
dense open subset
$U_{\nu_I}\subset\Hilb^{\nu_I}(\PP^2)$ made of
admissible cycles associated to $\nu_I$.
Recall that $Z=([Z_i])_{i\in I^*}\in U_{\nu_I}$ is such that $[Z_i]\in\Hilb^{\nu_i}(\PP^2)$ and $Z_i$ is a collection of $\nu_i$ points $p_{i,1},\ldots,p_{i,\nu_i}$ in $\PP^2$.
Let us define the map
\begin{equation}\label{eq:beta}
\beta_{\nu_I}\colon U_{\nu_I} \to \Gr(3,N),
\qquad
\beta_{\nu_I}(Z) = [H^0(\I_{Z_{\nu_I}}(d))],
\end{equation}
which sends $Z$ to the homaloidal net of plane curves of
degree $d$ with multiplicity $i$ at the points ${p_{i,j}}$, for each $i\in I^*$ and $j=1,\ldots,\nu_{i}$, cf.\ Definition \ref{def:Hilb_nuI}.

\begin{remark}
It is important to stress  that the base locus in $\Hilb^\rho(\PP^2)$ of a homaloidal net associated to a multi-index $\nu_I$ and supported on $r=\sum_{i=1}^{d-1}\nu_i$ distinct points determines a unique admissible cycle in $\Hilb^{\nu_I}(\PP^2)$ associated to $\nu_I$.
In that case, we say that such base locus in $\Hilb^\rho(\PP^2)$ is admissible.
Vice-versa an admissible cycle determines uniquely the homaloidal net.
\label{rem:bs}
\end{remark}

This, together with Lemma \ref{w}, yields
that the morphism $\beta_{\nu_I}$ is
a birational map onto its image in the
Grassmannian $\Gr(3,N)$ of planes in
$\PP^{N-1}=\PP(\C[x,y,z]_d)$. 

Next we want to go from the Grassmannian
to $\Biro_d$. This is done by
distinguishing a basis in the general point
of the  image of $\beta_{\nu_I}$.
To do this, choose three general $N-3$ planes $H_1,H_2,H_3$ in $\PP(\C[x,y,z]_d)=\PP^{N-1}$,  e.g.\ we may choose
\begin{align*}
  &
H_1={\langle y^d, y^{d-1}z\rangle}^\perp,
&&
H_2={\langle z^d, z^{d-1}x\rangle}^\perp,
&&
H_3={\langle x^d, x^{d-1}y\rangle}^\perp.
\end{align*}
This allows us to universally choose the
basis $W\cap H_1$, $W\cap H_2$, $W\cap
H_3$ for a general 3-dimensional linear
vector subspace $W\subset\C^N$. In other
words we have chosen three sections 
$\sigma_i:\Gr(3,N)\to U$ of the universal bundle over
the Grassmannian $\Gr(3,N)$.

\begin{definition}
In the above setting, define the rational map
\begin{align*}
\alpha_{\nu_I}: \PGL(3)\times U_{\nu_I} &\rto\Biro_d,
   \\
(\omega,Z) &\mapsto
[\ \omega(\sigma_1(\beta_{\nu_I}(Z))) :
\omega(\sigma_2(\beta_{\nu_I}(Z))) :
\omega(\sigma_3(\beta_{\nu_I}(Z)))\ ],
\end{align*}
where $\omega\in\PGL(3)$ is acting on
the 3-dimensional vector subspace
$W\subset\C^N$ as described in Lemma \ref{w}.
The map $\alpha_{\nu_I}$ is well-defined, in fact
$\alpha_{\nu_I}(\lambda\omega,Z)
=\alpha_{\nu_I}(\omega,Z)$, for any $\omega\in\PGL(3)$ and $\lambda\in\C^*$.
\end{definition}

\begin{lemma}\label{birational}
The map $\alpha_{\nu_I}$ is birational onto its image.
\end{lemma}

\begin{proof}
It is enough to prove that
$\alpha_{\nu_I}$ is generically
injective. But this follows immediately
by Lemma \ref{w} and Remark \ref{rem:bs}.
\end{proof}

\begin{lemma}\label{diffcomp}
Let $\nu_I=(\nu_1,\ldots,\nu_{d-1})$ and $\mu_I=(\mu_1,\ldots,\mu_{d-1})$ be two distinct admissible multi-indexes.
Then $\im(\alpha_{\nu_I})$ and $\im(\alpha_{\mu_I})$ lie in two different components of $\Biro_d$.
\end{lemma}

\begin{proof}
Let $r$ and $m$ be the two reduced
lengths of $\nu_I$ and $\mu_I$,
respectively.  If $r=m$ we conclude again by
Remark \ref{rem:bs}.
Assume that $r>m$. Then we have to prove
that
$\im(\alpha_{\nu_I})\not\supset\im(\alpha_{\mu_I})$.
This is equivalent to say that the
base locus of a general element in
$\im(\alpha_{\mu_I})$ cannot be obtained
as limit of base loci of general elements in
$\im(\alpha_{\nu_I})$. 

Assume that this is not the case. Then in
$\Hilb^\rho(\PP^2)$ there is a curve
whose general point represents the base
locus of an element in
$\im(\alpha_{\nu_I})$ and with a special
point associated to $Z_{\mu_I}$. In
other words we are saying that a bunch of points of multiplicity
$m_{i,1},\ldots, m_{i,h_i}$, with ordinary singularities, limits to a
point of some multiplicity $m_i$, with ordinary singularity, and this
is done in such a way that Noether's equations
(\ref{eq:Noether}) are always
satisfied. 

Fix  one point in the limit, say $p_1$ of
multiplicity $m_1$, and assume that $p_1$ is
the limit of $\{q_1,\ldots, q_{h_1}\}$ of
respective multiplicity $m_{1,j}$. The existence of the limit
forces
$$\frac{m_1(m_1+1)}2=\sum_j \frac{m_{1,j}(m_{1,j}+1)}2,
\qquad
\text{in particular}
\qquad
m_1\leqslant \sum_j m_{1,j},$$
with strict inequality if $h_1>1$.
This, together with Equations
(\ref{eq:Noether}) 
yields
$$d^2-1=\sum_i m_i^2=\sum_i\left(\sum_{j}
(m_{i,j}^2+m_{i,j})-m_i\right)=d^2-1+\sum_i(\sum_j m_{i,j})-m_i$$
hence we have the contradiction
$\sum_i(\sum_j m_{i,j})-m_i=0.$
\end{proof}

In the previous lemmata, we found an irreducible component of $\Biro_d$ for each admissible multi-index $\nu_I$.
In the next lemma, we show that each Cremona trasformation $\gamma\in\Biro_d$ belongs to one of these irreducible components.

\begin{lemma}\label{lem:limit}
Let $\gamma\in\Biro_d$ be a birational transformation, $Z^\gamma=\Bs(\gamma)$ its base locus and $\nu_I=(\nu_1,\ldots,\nu_{d-1})$ the corresponding  H-type.
Setting $\rho$ the length of $\nu_I$ and $r$ its reduced length,
there is a curve $C\in\Hilb^\rho(\PP^2)$ such that
\begin{itemize}
\item[$(i)$] $[Z^\gamma]\in C$,
\item[$(ii)$]  the general point $[Z_t]\in C$ is a zero-dimensional scheme supported on $r$ distinct points, with ordinary singularities, in $\PP^2$,
\item[$(iii)$] $Z_t$ is admissible.
\end{itemize}
\end{lemma}

\begin{proof}
Let $\cL_\gamma\subset|\cO(d)|$ be the homaloidal net
associated to $Z^\gamma$.
Consider the index $i(\gamma)$ introduced in Definition \ref{defgamma}.
If $i(\gamma)=0$,
the assertion is immediate, for any
degree $d$.
To conclude we argue  by induction
on $i(\gamma)$. Assume that $i(\gamma)=M>0$, and
let $p\in Z_{\rm red}$ a point of
multiplicity $m$ with infinitely near other base points.
Let $\omega:\PP^2\rto\PP^2$ be a quadratic transformation centered
in $p$ and two general points $q_1$ and
$q_2$, and
$\cL'=\omega_*\cL_\gamma$ the
strict transform linear system. Then
$\cL'\subset|\cO(2d-m)|$ is a
homaloidal net defining $\gamma'=\omega\circ\gamma$.
By construction, $i(\gamma')\leqslant i(\gamma)-1$ and by inductive
hypothesis we may describe its base locus $Z^{\gamma^\prime}=\Bs\cL'$ as limit of
admissible cycles with two ordinary points of
multiplicity $d$. Then applying
$\omega^{-1}$ we get the desired curve
in $\Hilb^\rho(\PP^2)$. 
\end{proof}

\begin{definition}
Let $\nu_I=(\nu_1,\ldots,\nu_{d-1})$ be an admissible multi-index.
We denote by $\Biro_{\nu_I}$ the irreducible component of $\Biro_d$ whose general element is defined by a homaloidal net of H-type $\nu_I$.
Moreover, we denote by $\Bir_{\nu_I}$ the intersection of the Zariski closure of $\Biro_{\nu_I}$ in $\PP^{3N-1}$ with $\Bir_d$.
\end{definition}

\begin{remark}
It is important to stress that a homaloidal net may degenerate to a linear system with a fixed component in such a way that the residual part is not a homaloidal net.
In particular the Zariski closure of $\Biro_{\nu_I}$ in $\PP^{3N-1}$ is not contained in $\Bir_d$.
The easiest occurrences of this behaviour are as follows.

Let $\cL$ be a general homaloidal net in $\Biro_{(4,1)}$ given by plane cubics with a double point $p_0$ and four simple base points $p_1,\ldots,p_4$. 
If we let $p_1,\ldots,p_4$ become aligned, then $\cL$ degenerates to a net with a fixed line and the residual part is composed with the pencil of lines through $p_0$.

Let $\cL$ be a general homaloidal net in $\Biro_{(3,3,0)}$ given by plane quartics with three double points $p_1,p_2,p_3$ and three simple base points. If we let $p_1,p_2,p_3$ become aligned, then $\cL$ degenerates to a linear system with a fixed line and the residual part is a 3-dimensional linear system of  cubics with six simple base points.
\end{remark}

The next lemma is classical, see e.g.\ \cite[p.\ 73]{Hu}, but the proof therein is not complete.

\begin{lemma}\label{notDeJ}
If $d\geqslant4$ and $\gamma\in\Biro_d$ is not a de Jonqui\`eres transformation, then the number of base points of the homaloidal net $\cL_\gamma$ is at most $d+2$.
\end{lemma}

\begin{proof}
By Lemma \ref{lem:limit} (cf.\ also Theorems \ref{th:classical} and \ref{th:classical_inverse}), the H-type $\nu_I$ of $\gamma$, and hence the number $r=\sum_{i=1}^{d-1}\nu_i$ of base points $\cL_\gamma$, does not depend on the position of the points in $\PP^2$.
Therefore, we may and will assume that these base points are in general position, in such a way that $\gamma$ can be factored as a composition of quadratic transformations, centered at base points of $\cL$ only, each one decreasing the degree of $\gamma$.

Now we proceed by induction on the degree $d$ of $\gamma$.

The base of induction is $d=4$.
In this case there is only one possible H-type, which is $(3,3,0)$, i.e.\ six base points and the assertion is trivially true.

Suppose then that $d>4$.
We set $m_1,m_2,m_3,\tilde d,\varepsilon,\mu_i$ according to Definition \ref{def:nuirr}, in particular $\tilde d=2d-m_1-m_2-m_3=d+\varepsilon$.
Recall that $r=\sum_{i=1}^{d-1} \nu_i$ and set $\tilde r=\sum_{i=1}^{\tilde d-1} \mu_i$.

Assume first that $\mu_I$ is the H-type of a de Jonqui\`eres, namely $\mu_I=(2\tilde d-2,0,\ldots,0,1)$ and $\tilde r=2\tilde d-1$.
Since $\nu_I$ is not the H-type of a de Jonqui\`eres and $\tilde d<d$, it follows that $\nu_I$ is obtained from $\mu_I$ by performing a quadratic transformation centered at points of multiplicity $e_j\leqslant 1$, say $1 \geqslant e_1 \geqslant e_2 \geqslant e_3 \geqslant 0$.
On the other hand, the point with multiplicity $e_3$ for $\mu_I$ would have multiplicity $m_3=e_3+c=\tilde d-e_1-e_2$ for $\nu_I$, that still have the point of multiplicity $\tilde d-1$ corresponding to the highest multiplicity point of $\mu_I$.
Since $m_1,m_2,m_3$ are chosen to be the highest multiplicities,
this is possible only if $e_2=e_3=0$.
It follows that $d=2\tilde d-e_1$, $m_1=\tilde d$ and $m_2=m_3=\tilde d-e_1$,
hence $\tilde r=2\tilde d-1=d+e_1-1$ and $r=\tilde r+3-e_1=d+2$, where $r$ is the number of base points of $\cL_\gamma$, that is the assertion.

Now we may and will assume that $\mu_I$ is not the H-type of a de Jonqui\`eres, so the inductive hypothesis says that $\tilde r\leqslant \tilde d+2$.

The construction of the $\mu_i$ implies that $r\leqslant \tilde r+3$ and one sees that equality holds if and only if  $m_1=m_2=m_3=-\varepsilon=d/2=\tilde d$.
Since $d>4$, one has $\tilde r\leqslant \tilde d+2=d/2+2\leqslant d-1$ and therefore $r=\tilde r+3\leqslant d+2$, that is the assertion.

Assume then $r\leqslant \tilde r+2$: one sees that equality holds if and only if $m_1=\tilde d>m_2=m_3=-\epsilon$ and $d=m_1+m_2$.
Since $\mu_I$ is not a de Jonqui\`eres, one has $\tilde d=m_1\leqslant d-2$ and therefore
$
r\leqslant \tilde r+2 \leqslant \tilde d+4 \leqslant d+2,
$
which is the assertion.

Finally, the remaining case is $r\leqslant \tilde r+1$. Since $\tilde d \leqslant d-1$, one has
\[
r\leqslant \tilde r+1 \leqslant \tilde d+3 \leqslant d+2,
\]
which concludes the proof of this lemma.
\end{proof}

We are finally ready for the proof of Theorem 1 in the introduction.

\begin{proof}[\textbf{Proof of Theorem 1.}]
Fix an integer $d\geqslant 2$.
By Lemma \ref{quasiproj}, $\Biro_d$ is a quasi-projective variety.
By Lemma \ref{lem:limit}, each irreducible component $\Biro_{\nu_I}$ of $\Biro_d$ is determined by an admissible and irreducible multi-index $\nu_I=(\nu_1,\ldots,\nu_{d-1})$ satisfying Equations \eqref{eq:Noether}.
By Lemma \ref{birational}, $\Biro_{\nu_I}$ is rational and it has dimension $8+\dim U_{\nu_I}=8+2r$, where $r=\sum_{i=1}^{d-1}\nu_i$ is the reduced length of $\nu_I$.
By Lemma \ref{DeJonq}, the maximum $r$ is $2d-1$ and it occurs for $\nu_I=(2d-2,0,\ldots,0,1)$, which gives an irreducible component of dimension $8+2(2d-1)=4d+6$, whose elements are de Jonqui\`eres transformations.
By Lemma \ref{notDeJ}, when $d\geqslant4$, the other irreducible components have dimension at most $8+2(d+2)=2d+12$.
\end{proof}

\begin{definition}\label{def:bira}
For any fixed positive integers $d$ and
$a<d$ let 
$$
\Bira{a}_{d} = \PP^{\bin{a+2}2-1}\times
\Bir_{d-a}.$$
As already observed in Remark \ref{r:disjoint}, there is a natural
inclusion of $\Bira{a}_{d}$ into
$\Bir_d$. For this we often identify
$\Bira{a}_{d}$ with its image in $\Bir_d$.
\end{definition}

\begin{remark} \label{birad} For 
integers $a<b<d$ we have, with
natural identifications,
$$\Bira{a}_{d}\cap\Bira{b}_{d}=
\PP^{\bin{a+2}2-1}\times\Bira{{b-a}}_{{d-a}}.
$$
\end{remark}

\begin{remark}
When $0<a<b<d$, the general element $[f_1:f_2:f_3]$ of $\tau_b(\Biro_{d-b})\subset\Bir_d$ cannot be the limit of elements in $\tau_{a}(\Biro_{d-a})\subset\Bir_d$ because $\gcd(f_1,f_2,f_3)$ is an irreducible polynomial of degree $b>a$.
\end{remark}

We are able to completely describe the behaviour of these varieties in low degrees.

\begin{example}\label{Bir2}
By Theorem 1, $\Biro_2=\Biro_{(3)}$ is irreducible of dimension 14.
By Remark \ref{r:disjoint}, $\Bir_2=\Biro_{(3)}\cup\, \Bira{1}_2$.
A general element $\gamma$ in $\Biro_{(3)}$ is given by a homaloidal net of conics with three distinct base points $p_1,p_2,p_3$.
If we let $p_3$ move to a general point of the line through $p_1$ and $p_2$, the line splits from the net and we get a degeneration of $\gamma$ to an element in $\Bira{1}_2$.
Since each element in $\Bira{1}_2$ can be obtained in this way, $\Bir_2$ is irreducible.
\end{example}

\begin{example}\label{Bir3}
Again by Theorem 1 and Remark \ref{r:disjoint}, $\Biro_3=\Biro_{(4,1)}$ is irreducible of dimension 18 and $\Bir_3=\Bir_{(4,1)}\cup\, \Bira{1}_3 \cup \Bira{2}_3$,
where $\Bira{2}_3=\tau_2(\Biro_1)$ has dimension 13
and $\Bira{1}_3$ contains $\tau_1(\Biro_2)$, which has dimension 16.

We claim that $\Bir_3$ has three irreducible components and it is connected.

By Remark \ref{birad}, $\Bira{1}_3 \cap \Bira{2}_3=\PP^2\times \Bira{1}_2$ has dimension 12.
On the other hand, a general element $\gamma$ in $\Biro_{(4,1)}$ is given by a homaloidal net
of cubics with a double point $p$ and four simple base points $q_1,\ldots,q_4$.
If we let $q_2$ move to a general point of the line through $p$ and $q_1$, the line splits from the net and we get a degeneration of $\gamma$ to an element in $\Bira{1}_3$. This shows that $\Bir_3$ is connected.

Note that any degeneration of an element in $\Biro_3$ has to contain a double point. Linear
systems of conics with a double point are homaloidal only in the presence of a fixed component. Then the
only possible degenerations of $\Biro_3$ are either a pair of fixed lines together with the linear
system of lines or a fixed line, say $l$, and a linear system of conics with a base point on $l$.
However, the general element in either $\Bira{1}_3$ or $\Bira{2}_3$ is such that there is no base
point, of the mobile part, lying on the fixed component and therefore it cannot be obtained as a limit of elements in $\Biro_3$.
This means that $\Bira{1}_3$ and $\Bira{2}_3$ give two further irreducible components, other than $\Bir_{(4,1)}$.
In \cite{CD}, it seems that the authors missed the component $\Bira{2}_3$ in $\Bir_3$.
\end{example}

\begin{example}
By Theorem 1, $\Biro_4$ has two irreducible components $\Biro_{(6,0,1)}$ and $\Biro_{(3,3,0)}$, having dimension respectively 22 and 20.

Reasoning as in the previous examples, one may check that $\Bir_4$ is connected (we will prove it in Theorem \ref{thm:connectBird} later) and its decomposition in irreducible components is
\[
\Bir_4=\Bir_{(3,3,0)} \cup \Bir_{(6,0,1)} \cup\, \Bira{1}_4\cup \Bira{2}_4\cup \Bira{3}_4,
\]
where the last three components have dimension respectively 17, 19 and 20.
\end{example}

To study the connectedness of $\Biro_d$ we need to understand
degenerations of base loci of homaloidal systems. This is an hard task
and almost nothing is known.
The only example we are aware of is the
one stated in \cite[V.III.25, pg.\ 231]{EC} of a quartic curve with three double points degenerating to a quartic with a triple point. This suggests that the component of $\Biro_4$ associated to the multi-index
 $(3,3,0)$ intersects the component of de Jonqui\`eres.

\begin{example}\label{ex:bir4}
Take the linear system of quartic curves with three double base points at $p_1=[0,0,1]$, $p_2=[t,0,1]$, $p_3=[0,t,1]$, $t\ne0$.
Its affine equation is
\begin{align*}
&
a_{{0}}{x}^{4}+a_{{1}}{x}^{3}y+a_{{2}}{x}^{2}{y}^{2}+a_{{4}}x{y}^{3}+a
_{{5}}{y}^{4}-2\,a_{{0}}t{x}^{3}+a_{{3}}{x}^{2}y+ \left( -a_{{4}}t+
a_{{1}}t+a_{{3}} \right) x{y}^{2}+
\\&
-2\,a_{{5}}t{y}^{3}+a_{{0}}{t}^{2}{x}^{2
}+ \left( -a_{{1}}{t}^{2}-ta_{{3}} \right) xy+a_{{5}}{t}^{2}{y}^{2}=0.
\end{align*}
For $t=0$, this is a linear system of quartics, whose general member is irreducible, with a triple base point at $[0,0,1]$ and three infinitely near simple base points in the direction of the lines $x=0$, $y=0$ and $x+y=0$, which are the limits of the lines $\overline{p_1p_3}:x=0$, $\overline{p_1p_2}:y=0$ and $\overline{p_2p_3}:x+y-tz=0$.

By imposing further three simple base points $p_4,p_5,p_6$ in general position, one gets a homaloidal net of type $(3,3,0)$ for a general $t\ne0$ and a homaloidal net of type $(6,0,1)$ for $t=0$. E.g., we choose $p_4=[1,-2,1]$, $p_5=[-2,1,1]$ and $p_6=[2,3,1]$ and we get the following Cremona transformation, for $t=1$,
\begin{align*}
[\,\,
&
\left( 3\,{x}^{3}-6\,{x}^{2}z+80\,x{y}^{2}-107\,xyz+3\,x{z}^{2}-9\,{y}^{3}-98\,{y}^{2}z+107\,y{z}^{2} \right) x
\\&
: -3\, \left( -{x}^{2}+10\,xy-12\,xz-{y}^{2}-12\,yz+13\,{z}^{2} \right) 
xy
\\&
: 3\, \left( -y+z \right)  \left( 12\,{x}^{2}-3\,xy-12\,xz-{y}^{2}+yz
 \right) y
\,\,],
\end{align*}
which has the following inverse map
\begin{align*}
[\,\, 
&
- \left( -36\,{x}^{2}-243\,xy+42\,xz-396\,{y}^{2}+116\,yz \right) 
 \left( -36\,xy+39\,xz-99\,{y}^{2}+107\,yz \right)
\\&
:
\left( -36\,xy+39\,xz-99\,{y}^{2}+107\,yz \right)  \left( 36\,xy-42\,
xz+99\,{y}^{2}-125\,yz+10\,{z}^{2} \right)
\\&
:
-108\,{x}^{3}z-1296\,{x}^{2}{y}^{2}+1539\,{x}^{2}yz-1152\,{x}^{2}{z}^{
2}-7128\,x{y}^{3}+10809\,x{y}^{2}z
+\\&
-6195\,xy{z}^{2}-30\,x{z}^{3}-9801\,
{y}^{4}+15840\,{y}^{3}z-8317\,{y}^{2}{z}^{2}-90\,y{z}^{3}
\,\,],
\end{align*}
while, for $t=0$, we get the following de Jonqui\`eres transformation
\begin{align*}
[\,\, 
&
- \left( 12\,{x}^{3}-217\,x{y}^{2}+308\,xyz-30\,{y}^{3}+308\,{y}^{2}z
 \right) x
:
6\, ( -2\,{x}^{2}-19\,xy+28\,xz
+\\&
-2\,{y}^{2}+28\,yz ) xy
:
6\, \left( -23\,{x}^{2}y+42\,{x}^{2}z-5\,x{y}^{2}+42\,xyz-2\,{y}^{3}
 \right) y
\,\,],
\end{align*}
which has the following inverse map
\begin{align*}
[\,\, 
&
-7\, \left( 6\,x+11\,y \right) ^{2} \left( -3\,y+2\,z \right)  \left( 
-6\,x-17\,y+4\,z \right)
:
14\, \left( -6\,x-17\,y+4\,z \right)  \left( 6\,x+11\,y \right) \cdot
\\&
 \left( -3\,y+2\,z \right) ^{2}
:
216\,{x}^{3}z-2484\,{x}^{2}{y}^{2}+4896\,{x}^{2}yz-1368\,{x}^{2}{z}^{2
}-9648\,x{y}^{3}+
\\&
16710\,x{y}^{2}z-5544\,xy{z}^{2}+96\,x{z}^{3}-9555\,{
y}^{4}+15942\,{y}^{3}z-5854\,{y}^{2}{z}^{2}+240\,y{z}^{3}
\,\,].
\end{align*}
These computations have been performed by using Maple.
\end{example}

\begin{proposition}
$\Biro_4=\Biro_{(6,0,1)} \cup \Biro_{(3,3,0)}$ is connected, $\dim (\Biro_{(6,0,1)} \cap \Biro_{(3,3,0)})=19$.
\end{proposition}

\begin{proof}
The connectedness follows from Example \ref{ex:bir4}, which shows that the general element in $\Biro_{(3,3,0)}$ may degenerate to a special element in the component $\Biro_{(6,0,1)}$ of de Jonqui\`eres.
The general choice of the base points of such de Jonqui\`eres is one triple point in the plane, three infinitely near simple points and further three simple points in the plane. Thus it lies in an open subset of $\PP^2\times\Sym^3\PP^1\times\Sym^3\PP^2$, that has dimension 11, plus 8 dimensions for the action of $\PGL(3)$ on the homaloidal net, cf.\ Lemma \ref{w}.
\end{proof}

\begin{proposition}
$\Biro_5=\Biro_{(8,0,0,1)} \cup \Biro_{(3,3,1,0)} \cup \Biro_{(0,6,0,0)}$ is connected.
\end{proposition}

\begin{proof}
The decomposition of $\Biro_5$ in three irreducible components $\Biro_{(8,0,0,1)}$, $\Biro_{(3,3,1,0)}$ and $\Biro_{(0,6,0,0)}$, having dimension respectively 26, 22 and 20, follows from Theorem 1.

One has $\Biro_{(3,3,1,0)}\cap\Biro_{(0,6,0,0)}\ne\emptyset$ for the same reason of the previous proposition, namely that a linear system of quintics with three double base points may degenerate to a linear system of quintics with a triple base point and three infinitely near simple base points. Using notation of Example \ref{ex:bir4}, one gets such a degeneration just by applying a quadratic transformation centered at $p_4,p_5,p_6$ to the degeneration of the linear system of quartics.

In order to prove the connectedness of $\Biro_5$, it is enough to show, with an example, that $\Biro_{(8,0,0,1)} \cap \Biro_{(3,3,1,0)} \ne \emptyset$.
We remark that, if we make collide a triple point and three double points, in general we get a quintuple point, not a quadruple one.
Thus we perform a special degeneration: we take the linear system $\cL$ of quintics with an \emph{oscnode} at $[0,0,1]$, along the direction of the conic $xz+y^2=0$, and a triple point at $[t,0,1]$, when $t\ne0$. The affine equation of $\cL$ is:
\begin{align*}
&
a_{{0}}{x}^{5}+a_{{1}}{x}^{4}y+a_{{2}}{x}^{3}{y}^{2}+a_{{3}}{x}^{2}{y}
^{3}+a_{{4}}x{y}^{4}+a_{{5}}{y}^{5}+
\\
&
-3\,a_{{0}}t{x}^{4}-2\,a_{{1}}t{x}^
{3}y+ \left( 2\,a_{{0}}{t}^{2}-a_{{2}}t \right) {x}^{2}{y}^{2}
+\left( a_{{1}}{t}^{2}+a_{{5}} \right) x{y}^{3}+
\\
&
-a_{{0}}{t}^{3}{y}^{4}+
3\,a_{{0}}{t}^{2}{x}^{3}+a_{{1}}{t}^{2}{x}^{2}y-2\,a_{{0}}{t}^{3}x{y}^
{2}-a_{{0}}{t}^{3}{x}^{2}
=0.
\end{align*}
When $t=0$, we get a linear system of quintics with a quadruple point whose general member is irreducible.
By imposing further three simple base points in general position, we get a homaloidal net $\subset\cL$ of type $(3,3,1,0)$ which degenerates to a homaloidal net defining a de Jonqui\`eres transformation. E.g., if we choose $[1,1,1]$, $[1,-1,1]$ and $[2,1,1]$ as simple base points, we get the following Cremona transformation, for $t=1$,
\begin{align*}
[\,\, &
10\,{x}^{2}{y}^{2}z+9\,x{y}^{3}z-18\,{y}^{3}{x}^{2}+5\,{y}^{4}x+9\,{y
}^{5}+5\,{x}^{5}-10\,x{y}^{2}{z}^{2}+15\,{x}^{3}{z}^{2}-15\,{x}^{4}z+
\\&
-5\,{x}^{2}{z}^{3}-5\,{y}^{4}z
:
y \left( 7\,{y}^{2}xz+2\,{y}^{4}-9\,{y}^{
2}{x}^{2}+5\,{x}^{4}-10\,{x}^{3}z+5\,{x}^{2}{z}^{2} \right)
\\&
: {y}^{2}
 \left( 4\,{y}^{3}+4\,yxz-5\,{x}^{2}z-8\,y{x}^{2}+5\,{x}^{3} \right)
\,\,]
\end{align*}
while, for $t=0$, we get the de Jonqui\`eres transformation
\begin{align*}
[\,\,  &  -60\,{y}^{3}{x}^{2}-5\,{y}^{4}x+30\,{y}^{5}+5\,{x}^{5}+30\,x{y}^{3}z
: 12\,{y}^{5}-29\,{y}^{3}{x}^{2}+5\,y{x}^{4}+12\,x{y}^{3}z
\\&
: 6\,{y}^{5}-5
\,{y}^{4}x-12\,{y}^{3}{x}^{2}+5\,{y}^{2}{x}^{3}+6\,x{y}^{3}z
\,\,].
\end{align*}
We are able to check the properties of these maps and to find their inverse maps by using Maple, in such a way we did in Example \ref{ex:bir4}.
\end{proof}

\begin{proposition}
$\Biro_6=\Biro_{(10,0,0,0,1)}\cup\Biro_{(3,4,0,1,0)}\cup\Biro_{(4,1,3,0,0)}\cup\Biro_{(1,4,2,0,0)}$ is connected.
\end{proposition}

\begin{proof}
The decomposition of $\Biro_6$ in four irreducible components follows from Theorem 1.
The usual degeneration of three double base points to one triple point with infinitely near three simple base points implies that $\Biro_{(4,1,3,0,0)}\cap\Biro_{(1,4,2,0,0)}\ne\emptyset$. It can be obtained by that of Example \ref{ex:bir4} by applying a quadratic transformation centered at $p_4$, $p_5$ and at a general point in the plane.

To conclude we show, with two examples, that
$$\Biro_{(3,4,0,1,0)}\cap\Biro_{(1,4,2,0,0)}\ne\emptyset\  {\rm and} \ \Biro_{(3,4,0,1,0)}\cap\Biro_{(10,0,0,0,1)}\ne\emptyset.$$

First example: take the linear system of sextics with a triple base point at $[0,0,1]$, with infinitely near a double base point in the direction of the line $x=0$, and another triple base point at $[t,0,1]$, $t\ne0$. When $t=0$, we get a linear system of sextics with a quadruple base point with infinitely near a double base point.
By imposing further three double base points and a simple base point, e.g.\ we choose  $[1,1,1]$, $[-1,1,1]$, $[2,1,1]$ and $[2,-3,1]$ respectively, we get a homaloidal net of type $(1,4,2,0,0)$ for a general $t\ne0$ and a homaloidal net of type $(3,4,0,1,0)$ for $t=0$. In particular, for $t=1$, we get the map
\begin{align*}
[\,\,
&
27\,{x}^{6}-216\,{y}^{6}-81\,{x}^{2}y{z}^{3}+135\,{x}^{3}y{z}^{2}-108
\,{x}^{2}{y}^{2}{z}^{2}-368\,x{y}^{3}{z}^{2}-27\,{x}^{4}yz
+\\&
+108\,{x}^{3}{y}^{2}z+520\,x{y}^{4}z-27\,{x}^{3}{z}^{3}+81\,{x}^{4}{z}^{2}-81\,{x}
^{5}z+324\,{y}^{5}z-27\,{x}^{5}y-260\,x{y}^{5}
\\&
: 3\,xy \left( -z+y
 \right)  \left( -4\,{y}^{3}+22\,{y}^{2}z-18\,x{y}^{2}+9\,x{z}^{2}-18
\,{x}^{2}z+9\,{x}^{3} \right)
:
9\,{y}^{2} \left( -z+y \right)\cdot  
\\&
\left( 6\,{y}^{3}+4\,x{y}^{2}-7\,yxz-3\,{x}^{2}y+3\,{x}^{3}-3\,{x}^{2}z
 \right)
\,\, ]
\end{align*}
and, for $t=0$, we get the map
\begin{align*}
[\,\,
&
27\,{x}^{6}-540\,{y}^{6}-430\,x{y}^{3}{z}^{2}+54\,{x}^{4}yz+216\,{x}^
{3}{y}^{2}z-702\,{x}^{2}{y}^{3}z+644\,x{y}^{4}z
+\\&
+648\,{y}^{5}z-108\,{x}
^{5}y+513\,{x}^{2}{y}^{4}-322\,x{y}^{5}
:
3\,y \left( -z+y \right) 
 ( 36\,{y}^{4}+20\,x{y}^{3}
+\\&
-20\,{y}^{2}xz-45\,{x}^{2}{y}^{2}+9\,{
x}^{4} ) 
:
9\,{y}^{2} \left( -z+y \right)  \left( 6\,{y}^{3}+2\,x
{y}^{2}-5\,yxz-6\,{x}^{2}y+3\,{x}^{3} \right)
\,\, ].
\end{align*}

Second example: take the linear systems of sextics with a double base point at $[0,0,1]$ with other three infinitely near double base points, each one infinitely near to the previous one, along the conic $xz+y^2=0$ and a quadruple base point at $[t,0,1]$, $t\ne0$.
When $t=0$, we get a linear system of sextics with a base point of multiplicity 5.
By imposing further three simple base points, e.g.\ we choose $[1,1,1]$, $[-1,1,1]$, and $[2,1,1]$, we get a homaloidal net of type $(3,4,0,1,0)$ for a general $t\ne0$ which degenerates to a homaloidal net of de Jonqui\`eres type $(10,0,0,0,1)$. In particular, for $t=1$, we get the map
\begin{align*}
[\,\,
&
-2\,{x}^{6}-8\,{y}^{6}-4\,x{y}^{2}{z}^{3}+8\,{x}^{2}{y}^{2}{z}^{2}-4
\,{x}^{3}{y}^{2}z+5\,{x}^{2}{y}^{3}z-8\,x{y}^{4}z-2\,{x}^{2}{z}^{4}+8
\,{x}^{3}{z}^{3}
\\&
-12\,{x}^{4}{z}^{2}-2\,{y}^{4}{z}^{2}+8\,{x}^{5}z-5\,{
x}^{3}{y}^{3}+13\,{x}^{2}{y}^{4}+5\,x{y}^{5}
:
-2\,y \left( -xz+{x}^{2}+
xy-{y}^{2} \right)  \cdot 
\\&
\left( -xz+{x}^{2}-{y}^{2} \right)  \left( -z+x-y
 \right)
:
-{y}^{2} \left( -2\,xz+2\,{x}^{2}-xy-2\,{y}^{2} \right) 
 \left( -xz+{x}^{2}+xy-{y}^{2} \right)
\,\, ]
\end{align*}
and, for $t=0$, we get the de Jonqui\`eres transformation
\begin{align*}
[\,\,
&
-3\,{x}^{6}-20\,{y}^{6}-20\,x{y}^{4}z+20\,{x}^{3}{y}^{3}+23\,{x}^{2}{
y}^{4}
:
-y ( 8\,{y}^{5}+8\,{y}^{3}xz+3\,{x}^{5}+
\\&
-11\,{x}^{3}{y}^{2}-8\,{x}^{2}{y}^{3} )
:
-{y}^{2} \left( 4\,{y}^{4}+4\,xz{y}^{2}+3
\,{x}^{4}-4\,{x}^{3}y-7\,{x}^{2}{y}^{2} \right) 
\,\,].
\end{align*}
Again we checked the properties of these maps and we found their inverse maps by using Maple, in such a way we did in Example \ref{ex:bir4}.
\end{proof}

\begin{remark}
  We want to stress a difference between $\Biro_5$ and $\Biro_6$. 
It is not difficult to prove that any pair of components in $\Biro_5$ intersects. The situation for $\Biro_6$ is, quite unexpectedly, different.
We claim that $\Biro_{(4,1,3,0,0)}\cap\Biro_{(10,0,0,0,1)}=\emptyset$.
Let $\PP^2\times\Delta\to \Delta$  be a degeneration, over a complex
disk $\Delta$, with a linear system $\cL$ such that the map induced by
$\cL_0$ is in $\Bir_{(10,0,0,0,1)}$ and the map induced by $\cL_t$ is
in $\Biro_{(4,1,3,0,0)}$, for $t\in\Delta\setminus\{0\}$. Let
$\mu:Y\to\PP^2\times\Delta$ be the blow up of the unique singular
point $p$ in $\cL_0$, with exceptional divisor $E\simeq\PP^2$.  
Note that the sections associated to the singular points of $\cL_t$ has to intersect in $p$, because $p$ is the unique singular point in $\cL_0$. Hence, the strict transform $\cL^Y$ is such that $\cL^Y_{0|E}$ is a curve with degree equal to the multiplicity of $\cL_0$ at $p$, that is 5, and with three triple points and a double point. This forces, by a direct computation, $\cL^Y_{0|E}$ to be non reduced and therefore introduces a fixed component in $\cL_0$. Hence the map induced by $\cL_0$ is in $\Bir_{(10,0,0,0,1)}\setminus\Biro_{(10,0,0,0,1)} $ and $\Biro_{(4,1,3,0,0)}\cap\Biro_{(10,0,0,0,1)}=\emptyset$.
\end{remark}

The connectedness of $\Bir_d$ is considerably simpler even if not all irreducible components intersect each other.

\begin{remark}
If $\nu_I=(2d-2,0,\ldots,0,1)$, i.e.\ if $\nu_I$ is the H-type of de Jonqui\`eres transformations, then it is easy to check that $\Bir_{\nu_I}$ meets $\Bira{a}_d$, for each $a=1,\ldots,d-1$.

However, the same statement does not hold for each admissible multi-index $\nu_I$.

For example, one may check that the minimum $d$ such that there exists an admissible $\nu_I=(\nu_1,\ldots,\nu_{d-1})$ with $\Bir_{\nu_I}\cap\,\Bira{1}_d=\emptyset$ is $d=10$. Moreover, there are exactly two such admissible $\nu_I$ of degree $d=10$, namely $(0,0,7,0,0,1,0,0,0)$ and $(3,0,0,6,0,0,0,0,0)$.
\end{remark}

We already
recognized that $\Bir^a_d\cap\Bir^b_d\neq\emptyset$. Hence to conclude
the connectedness it is enough to show that for any admissible index
$\nu_I$ there is a degeneration with a fixed component. 
The following lemmata allow us to produce these degenerations.
 
\begin{lemma}\label{lem:fixed}
  Let $\chi:\PP^2\times\Delta\to\Delta$ be a one dimensional family
   over a complex disk $\Delta$. Let $F_t$ be the fiber over the point $t\in\Delta$.
Let $C_1$, $C_2$ and $C_3$ be three disjoint sections  of
$\chi$. Let $p^i_t:=C_i\cap F_t$ be the intersection of the $i^{\rm
  th}$-section with the fiber $F_t$, $i=1,2,3$. Assume that the points $p^i_t$ are in general position for any $t$. Then there is a birational
modification
$\Omega:\PP^2\times\Delta\rto\PP^2\times\Delta$ such that
$\Omega_t:=\Omega_{|F_t}$ is the quadratic Cremona transformation centered at
the points $\{p_t^i\}$.
  Assume that there is a linear system $\cH\in \textrm{Pic}(\PP^2\times\Delta)$
  such that $\cH_t$ is a homaloidal net associated to a multi-index
  $\nu_I$ and $\cH_0$ is a homaloidal net associated to a multi-index
  $\mu_I$ and 
$$\sum_i \mult_{Z_{\nu_I}}p_t^i <\sum_i \mult_{Z_{\mu_J}}p_0^i.$$
Let $\cH^\prime:=\Omega_*\cH$ be the transformed linear system. Then $\cH^\prime_{|F_0}$ has a fixed component.  
\end{lemma}
\begin{proof}
  Let $D_{ij}$ be the
  divisor covered by lines spanning the points $p_t^i$ and $p_t^j$
  inside $F_t$. The general position assumption ensures that $D_{ij}$ is a smooth minimally ruled surface. Let $\phi:Y\to\PP^2\times\Delta$ be the blow up of $\PP^2\times\Delta$ along the disjoint
  sections $C_i$ with exceptional divisors $E_i$. Let $D_{ij}^Y$ be
  the strict transform of $D_{ij}$ on $Y$ and $l^{ij}_t$ the strict
  transform of the line $\langle p_t^i,p_t^j\rangle\subset F_t$, and
  $F_t^Y$ the strict transform of the fiber $F_t$. Then we have the
  following intersection numbers 
$$ K_Y\cdot l_t^{ij}= K_{F_t^Y}\cdot l_t^{ij}=-1$$
and
$$ D^Y_{ij}\cdot l_t^{ij}= D_{ij}^Y\cdot D_{ij}^Y\cdot
F_t^Y=(l_t^{ij}\cdot l_t^{ij})_{F_t^Y}=-1.$$
Moreover $D_{ij}^Y$ is ruled by $l_t^{ij}$ and all fibers are
irreducible and reduced.
This shows, by Mori theory see for instance \cite[Theorem 4.1.2]{AM}, that $l_t^{ij}$ spans an extremal
ray and the extremal ray can be contracted to a smooth curve $Z_{ij}$ in a smooth 3-fold. Let $\psi$ be the blow
down of the three disjoint divisors $D_{ij}$. Then
$\psi$ is a morphism from  $Y$ to $\PP^2\times\Delta$. The required map
 $\Omega$ is just $\psi\circ\phi^{-1}$. 
To conclude observe that for the general fiber
$\Omega_t(\cH_t)=\cH^\prime_{|F_t}$ and the $\deg
\Omega_t(\cH_t)=\deg\cH^\prime=2d-\sum \mult_{Z_{\nu_I}}p^t_i$. The
numerical assumption on the multiplicities forces
$\deg\Omega_0(\cH_0)<\deg\cH^\prime$. This yields a fixed component in $\cH^\prime_{|F_0}$.
\end{proof}
\begin{remark} The usage of the above Lemma is to produce
  degenerations with fixed components starting from known degeneration
  in a different pure degree.
\end{remark}

To apply the above Lemma we have to construct degenerations. This is
the aim of the next Lemma.

\begin{lemma}\label{lem:near}
  Let $\nu_I$ be an admissible multi-index in degree
  $d$.  Let $Z_{\nu_I}$ be a base locus of a general homaloidal net associated to the multi-index $\nu_I$, and $p_1$, $p_2\in Z_{\nu_I}$ two points. Assume that 
$$m_1:=\mult_{Z_{\nu_I}}p_1\leqslant \mult_{Z_{\nu_I}}p_2=:m_2.$$
Then there is a degeneration $\chi:\PP^2\times\Delta\to\Delta$ and a base scheme $\mathcal Z$ such that  $Z_t:={\mathcal Z}_{|F_t}$ is associated to the multi-index $\nu_I$ and 
$Z_0:={\mathcal Z}_{|F_0}$ has the point $p_1$ infinitely near to $p_2$.
\end{lemma}
\begin{proof}
  The  multiplicity of $p_2$ is at least the
  one of $p_1$ and we may degenerate $p_1$ into $p_2$. Assume that
  $\mult_{p_1}Z_{\nu_I}=m_1$ and $\mult_{p_2}Z_{\nu_I}=m_2$ then a local equation
  of such a degeneration can be as follows
$$tx_0^{d-m_1}p+x_1^{d-m_2}(x_2^{m_1}h+tg)+x_2^d=0,$$
where $g\in \C[x_0,x_1,x_2]$ is such that for $t\neq 0$ the points $[1,0,0]$ and $[0,1,0]$ are ordinary points of multiplicities $m_1$ and $m_2$ respectively, and for $t=0$ the point $[0,1,0]$ is of multiplicity $m_2$ with an infinitely near point of multiplicity $m_1$.
\end{proof}

\begin{theorem} \label{thm:connectBird}
  The quasi-projective variety $\Bir_d$ is connected.
\end{theorem}
\begin{proof} As already observed we have only to prove that for any admissible index $\nu_I$ the general element admits a degeneration with a fixed component. Let $\cL$ be a general homaloidal net associated to $\nu_I$ and $\omega$ a standard Cremona centered in three points of $\Bs\cL$ that lowers the degree. Let $\cL^\prime$ be the transformed homaloidal net and $q_1,q_2,q_3$ the three  points of indeterminacy of $\omega^{-1}$. Then by Noether--Castelnuovo Theorem the $q_i$'s are not of maximal multiplicity. That is we may assume that there is a point $x\in\Bs\cL^\prime$ with $\mult_x\cL^\prime>\mult_{q_1}\cL^\prime$.  Then by  Lemma \ref{lem:near} there is a degeneration $\chi:\PP^2\times\Delta\to\Delta$ and a base scheme $\mathcal Z$  such that  $Z_t:={\mathcal Z}_{|F_t}$ is associated to the multi-index $\nu_I$ and 
$Z_0:={\mathcal Z}_{|F_0}$ has the point $q_1$ infinitely near to $x$.

 Let $C_i$ be the section of $\chi$ associated to the point $q_i$ and $\Omega$,  $\cH^\prime$ the birational modification and linear system on $\PP^2\times\Delta$ as in Lemma \ref{lem:fixed}. Then we may apply Lemma  \ref{lem:fixed} to produce a Cremona transformation $\Omega^{-1}:\PP^2\times\Delta\rto\PP^2\times \Delta$ that induces $\omega^{-1}$ on the general fiber and  produces a fixed component in the special linear system $\cH_{|F_0}$. In particular this produces a degeneration of $\cL$ to a homaloidal net with a fixed component.
\end{proof}

We are now ready to complete the proof of Theorem 3.

\begin{proof}[\textbf{Proof of Theorem 3.}]
By Remark \ref{r:disjoint}, $\Bir_d$ is the union of $\Biro_d$ and $\Bira{a}_{d}$, for each $a=1,\ldots,d-1$.
Note that $\Bira{{d-1}}_{d}$, that is $\tau_{d-1}(\Biro_1)$, has dimension $8+\binom{d+1}{2}-1=d(d+1)/2+7$, for $d\geqslant2$.
By Theorem 1, irreducible components of $\Bir_d$ coming from irreducible components of $\Biro_{d-a}$, $a\leqslant d-2$, have dimension at most
\[
4(d-a)+6+a(a+3)/2=4d+6+a(a-5)/2\leqslant d(d-1)/2+13.
\]
When $d\geqslant7$, this implies that $\Bira{a}_{d}$, $a\leqslant d-2$, has components of smaller dimension.
\end{proof}

\appendix
\section{Test if a multi-index is irreducible, i.e.\ admissible}

The following script defines a function ``adm'' in PARI/GP that checks if a multi-index $\nu_I$ is irreducible, or equivalently admissible, cf.\ Definition \ref{def:nuirr} and Theorem \ref{th:classical_inverse}.

The input is a vector $\nu_I=[\nu_1,\nu_2,\ldots,\nu_{d-1}]$, where the square brackets are used in GP for denoting row vectors.
The output of adm$(\nu_I)$ is either $1 = $ true, i.e.\ $\nu_I$ is admissible, or $0 = $ false, i.e.\ $\nu_I$ is not admissible.

\begin{verbatim}
adm(v) = { local( d=1+#v , s=1-(#v+1)^2 , m=vector(3) , t=0 , e=0 );
  for( i=1,#v , s = s+i^2*v[i] );
  if( s , print("ERROR: the self-intersection is not 1"); return(0) );
  s = -#v*3;
  for( i=1,#v , s = s+i*v[i] );
  if( s , print("ERROR: the genus is not 0"); return(0) );
  while( d>2,
    s = 0; t = d; e = d;
    for( i=1,3 , while( !s, t=t-1; s=s+v[t] );
       m[i]=t; e=e-t; v[t]=v[t]-1 ; s=s-1 );
    if( m[1]+m[2]>d, print("The net is reducible"); return(0) );
    for( j=1,3 , t=m[j]+e; if( t , v[t] = v[t] +1 ) );
  d = d+e );
  1}
\end{verbatim}

For example, after defining this function ``adm'' in GP, the command ``adm$([0,6,0,0])$'' returns ``1'', whilst the command ``adm$([6,0,2,0])$'' prints ``The net is reducible'' and returns ``0'', cf.\ Remark \ref{counterexample}.

\end{document}